\documentclass[12pt]{article}
\usepackage{standalone}
\usepackage{amsmath, amsthm, amssymb, color}
\usepackage[colorlinks=true,linkcolor=blue,urlcolor=blue]{hyperref}
\usepackage{graphicx}
\usepackage{caption}
\usepackage{mathtools}
\usepackage{enumerate}
\usepackage{verbatim}
\usepackage{tikz,tikz-cd,tikz-3dplot}
\usepackage{amssymb}
\usetikzlibrary{matrix}
\usetikzlibrary{arrows}
\usepackage{algorithm}
\usepackage[noend]{algpseudocode}
\usepackage{caption}
\usepackage[normalem]{ulem}
\usepackage{subcaption}
\usepackage{multicol}
\oddsidemargin \evensidemargin
\usepackage{makecell}
\usepackage{array}
\usepackage{enumitem}
\usepackage{booktabs}
\usepackage{amsmath}

\newtheorem{theorem}{Theorem}[section]

\newtheorem{proposition}[theorem]{Proposition}
\newtheorem{lemma}[theorem]{Lemma}

\theoremstyle{definition}
\newtheorem{definition}[theorem]{Definition}

\newtheorem{question}[theorem]{Question}

\newtheorem{remark}[theorem]{Remark}
\newtheorem{notation}[theorem]{Notation}

\newtheorem{examplex}[theorem]{Example}
\newenvironment{example}
{\pushQED{\qed}\examplex}
{\popQED\endexamplex}

\theoremstyle{plain}
\makeatletter
\newtheorem*{rep@theorem}{\rep@title}
\newcommand{\newreptheorem}[2]{%
	\newenvironment{rep#1}[1]{%
		\def\rep@title{#2 \ref*{##1}}%
		\begin{rep@theorem}}%
	{\end{rep@theorem}}}
\newreptheorem{theorem}{Theorem}
\makeatother

\theoremstyle{definition}
\newtheorem*{rep@definition}{\rep@title}
\newcommand{\newrepdefinition}[2]{%
	\newenvironment{rep#1}[1]{%
		\def\rep@title{#2 \ref*{##1}}%
		\begin{rep@definition}}%
	{\end{rep@definition}}}
\newrepdefinition{definition}{Definition}

\usepackage{tabularx}
\usepackage{array}

\newcommand{\bC}{\mathbb C}
\newcommand{\bZ}{\mathbb Z}

\newcommand{\bP}{\mathbb P}

\newcommand{\cC}{\mathcal C}

\newcommand{\cH}{\mathcal H}

\newcommand{\cO}{\mathcal O}
\newcommand{\cQ}{\mathcal Q}
\newcommand{\cS}{\mathcal S}

\newcommand{\cV}{\mathcal V}
\newcommand{\cW}{\mathcal W}
\newcommand{\cY}{\mathcal Y}

\renewcommand{\phi}{\varphi}

\newcommand{\gen}[1]{\langle {#1} \rangle}
\DeclareMathOperator{\im}{im}
\DeclareMathOperator{\Hom}{Hom}

\DeclareMathOperator{\Gr}{Gr}
\DeclareMathOperator{\Fl}{Fl}

\DeclareMathOperator{\rk}{rk}
\DeclareMathOperator{\Ann}{Ann}

\DeclareMathOperator{\codim}{codim}
\newcommand{\Con}{\mathrm{Con}}
\newcommand{\CL}{\mathrm{CL}}

\DeclareMathOperator{\sHom}{\mathcal{H}\kern -1pt\text{om}}
\newcommand{\crk}{\operatorname{crk}}

\usepackage{xcolor}

\usepackage[breakable,skins]{tcolorbox}
\newtcolorbox{tbox}[1][]{%
	breakable,
	enhanced,
	colframe=black,
	coltitle=white,
	#1
}

\title{\bf Conormal Rank of Subvarieties of Grassmannians}
\author{Elizabeth Pratt}
\date{}
\begin{document}
	\maketitle

	\begin{abstract}
	The cotangent space to any subvariety of a Grassmannian is naturally identified with a space of linear homomorphisms. We use this to define a new statistic on a subvariety of a Grassmannian, called \emph{conormal rank}. We show that Chow--Lam forms and their natural generalizations have low conormal rank, and moreover that hypersurfaces of sufficiently low conormal rank are all of this form, generalizing a result of Gelfand--Kapranov--Zelevinsky. We also compute the conormal rank of many families of varieties, including Schubert varieties, torus orbit closures, and positroid varieties.
	\end{abstract}

	\section{Introduction}

	The complex Grassmannian $\Gr(k,n)$ is a projective variety whose points parameterize $k$-dimensional subspaces of $\bC^n$. It is a well-known fact that the cotangent space to  $\Gr(k,n)$ at each point $L$  can be naturally identified with the space of linear homomorphisms from $\bC^n / L$ to $L$. Thus one may talk about the rank of a  cotangent vector as a homomorphism. We use this to define a new statistic for each subvariety of $\Gr(k,n),$ called \emph{conormal rank}. It is, in words, the rank of a general cotangent vector.

	\begin{definition}\label{def:rank}
		Let $\cY \subset \Gr(k,n)$ be an irreducible variety. Let $N^*_{\cY}\Gr(k,n)$ denote the conormal bundle above the regular locus. The \emph{conormal rank} $\crk(\cY)$ is the minimum nonnegative integer $r$ such that every cotangent vector $\eta \in N^*_{\cY}\Gr(k,n)$ has rank at most $r.$
	\end{definition}

	Varieties of  conormal rank one are called \emph{strongly coisotropic}, a term coined by Kohn and Mathews \cite{KohnMathews}. Their work generalized the earlier \emph{coisotropic hypersurfaces} of Gelfand, Kapranov, and Zelevinsky in \cite{gkz}, defined as hypersurfaces of conormal rank one. That is, the unique cotangent vector at every regular point is a homomorphism of rank one.\footnote{The original definition requires the conormal vector $\eta_L$ at every regular point $L$ to have rank exactly one. However, this is equivalent; at each regular point the conormal space is spanned by a nonzero vector.}
	An important instance of a coisotropic hypersurface is the classical Chow hypersurface of an embedded projective variety, defined by Chow and van der Waerden \cite{CvW}.

	\begin{definition}\label{def:chowlocus}
		Let $\cV \subset \bP^{n-1}$ be an irreducible variety of dimension $d.$ The \emph{Chow locus} of $\cV$ is defined as
		\[\cC_\cV := \overline{\{L \ : \ L \cap \cV \neq \varnothing\}} \ \subset \ \Gr(n-d-1,n).\]
		It is a hypersurface, and its defining equation ${\rm C}_\cV$ is called the \emph{Chow form} of $\cV.$
	\end{definition}

	\begin{example}[The twisted cubic]
		Let $\mathcal{V}$ be the closure of $(1:t:t^2:t^3)$ in $\bP^3$. Its Chow form in primal Pl\"ucker coordinates  is the determinant of the  {\em B\'ezout matrix}:
		\begin{equation}
			\label{eq:bezout}
			{\mathrm C}_\mathcal{V} \,\, = \,\, \det
			\begin{bmatrix}
				\,p_{12} & p_{13} & p_{14}  \,\\
				\,p_{13} & p_{14} + p_{23} & p_{24} \,  \\
				\,p_{14} & p_{24} & p_{34}\,
			\end{bmatrix}.
		\end{equation}
		Working locally in an affine chart $p_{12} \neq 0,$ we may represent points in the Grassmannian as rowspans of matrices of the form
		\[M = \begin{bmatrix}
			1 & 0 & a & b \\
			0 & 1 & c & d
		\end{bmatrix}.\]
		Then the Jacobian of the determinant, when arranged into a matrix, becomes \renewcommand{\arraystretch}{1.4}
		\[
		J = \begin{bmatrix}
			\frac{\partial f}{\partial a} & \frac{\partial f}{\partial b}\\
			\frac{\partial f}{\partial c} & \frac{\partial f}{\partial d}
		\end{bmatrix}
		=
		\begin{bmatrix}
			-c^2d+bc-2ad+2d^2 & c^3+ac-3cd-2b \\
			3bc^2-2acd+ab-3bd & -ac^2-a^2-3bc+4ad-3d^2
		\end{bmatrix}.
		\]
		\renewcommand{\arraystretch}{1}
		The determinant of $J$ is divisible by the determinant \eqref{eq:bezout} after $a, b, c, d$ are substituted for the Pl\"ucker variables. Thus the rank of each conormal vector is at most one on the intersection of the Chow hypersurface with this affine chart.
	\end{example}

	Chow forms can be generalized in two natural ways. \emph{Associated hypersurfaces}, introduced in \cite{gkz}, generalize Definition \ref{def:chowlocus} by increasing the dimension of the linear space $L$ in $\bP^{n-1}$ and asking for it to be tangent to $\cV$. This recovers other constructions in classical projective geometry; for example, the dual variety of $\cV$ parameterizes hyperplanes tangent to $\cV.$ Associated hypersurfaces have been studied in computer vision \cite{vision} and elsewhere in the sciences. Their degrees coincide with \emph{polar degrees} \cite{holme}, which are important invariants in metric algebraic geometry; see \cite[Chapter 4]{metricag}.

	Another generalization of the Chow form begins instead with a variety $\cV$ embedded in a Grassmannian $\Gr(k,n),$ where $k$ may be greater than one. The Chow--Lam form, defined in \cite{ChowLam}, captures incidence relations between $\cV$ and certain subGrassmannians. Chow--Lam forms were first studied by Thomas Lam in the context of scattering amplitudes, where the subvarieties in question are positroid varieties \cite[Section 19]{Lam}.

	One can do both generalizations simultaneously, fixing a variety in $\Gr(k,n)$ and asking which subGrassmannians $\Gr(k,Q)$ are tangent to it at a regular point. This yields \emph{principal Chow--Lam loci}\footnote{These were defined in \cite{ChowLam}, under the name \emph{higher Chow--Lam loci}; we change the name with permission of the authors and give further justification in Remark \ref{rmk:principal-vs-higher}.}, which we define presently. We summarize the constructions in Table \ref{tab:history}.

\begin{table}[h]
\centering
\renewcommand{\arraystretch}{1.3}
\begin{tabular}{|c|c|c|}
	\hline
 & $k = 1$ & $k \geq 1$ \\
\hline
 \shortstack{incidence \\ [4pt]$(\dim \Gr(k,Q) = \codim \cV - 1)$} & \rule{0pt}{7ex}\shortstack{Chow locus \cite{CvW}\\ [4pt] 1937} & \rule{0pt}{7ex}\shortstack{Chow--Lam locus \cite{ChowLam} \\ [4pt]2025} \\
\hline
\shortstack{tangency \\ [4pt]$(\dim \Gr(k,Q) \geq \codim \cV - 1)$} & \rule{0pt}{7ex}\shortstack{Associated hypersurface \cite{gkz}\\ [4pt] 1994} & \rule{0pt}{7ex}\shortstack{Principal Chow--Lam locus \cite{ChowLam} \\ [4pt] 2025} \\
\hline
\end{tabular}
\caption{Different conditions on subGrassmannians w.r.t. a fixed variety $\cV \subseteq \mathrm{Gr}(k,n)$.}\label{tab:history}
\end{table}


\begin{definition}\label{def:principal cl}
Let $\cV\subset\Gr(k,n)$ be an irreducible variety of dimension
$k(n-\ell)-1+i$, where $k< \ell < n$ and $0\leq i\leq k(\ell - k).$ Define the incidence variety
\[
\Phi_{\cV}^i
=
\overline{
\left\{
(L,Q):
L\in\cV_{\rm reg}
\text{ and }
\Gr(k,Q)\cap\cV
\text{ is not transverse at }L
\right\}} \, \subset \, \Fl(k,\ell,n).
\]
The $i$th \emph{principal Chow--Lam locus} $\mathcal{PCL}_{\cV}^i$ is the projection of $\Phi_{\cV}^i$ to $\Gr(\ell,n).$ If $\mathcal{PCL}_{\cV}^i$ is a hypersurface, its equation is
called the $i$th \emph{principal Chow--Lam form} and denoted ${\rm PCL}^i_\cV$.
\end{definition}

A major structural result of Gelfand, Kapranov and Zelevinsky is that all associated hypersurfaces have rank one, and moreover that \emph{every} rank one hypersurface is an associated hypersurface of a projective variety \cite[Theorem 4.3.14]{gkz}. The second fact is perhaps more surprising, because one must reconstruct the projective variety $\cV$ from the associated hypersurface. In this article we prove analogous results for the principal Chow--Lam loci.

\begin{theorem}\label{thm:rankk}
Let $\cV\subset\Gr(k,n)$ satisfy the dimension hypothesis in
Definition~\ref{def:principal cl}. If $\mathcal{PCL}^i_{\cV}$ is a
hypersurface, then its conormal rank is at most $\min(k, n-\ell).$
\end{theorem}

The converse statement is more subtle. The difficulty is that principal Chow--Lam loci, except for the original Chow--Lam case, may be reducible. This contrasts with associated hypersurfaces, which are all irreducible by a trick of Cayley; see \cite[Proposition 5]{kohn}. To remedy this, we introduce the \emph{Chow--Lam discriminant} (Definition \ref{def:cl discriminant}), which is an irreducible factor of the principal Chow--Lam form. The naming convention is chosen to mimic the principal $A$-determinant and the $A$-discriminant from toric geometry. The analogous result to the GKZ result is that every hypersurface of conormal rank $k$ is a Chow--Lam discriminant of a subvariety of $\Gr(k,n).$

\begin{theorem}\label{thm:cl recovery}
Let $\cH \subset \Gr(\ell,n)$ be a rank $k$ hypersurface, where $k < \ell.$ Let $\cV := \pi(\Lambda_\cH)$ be the projection of its conormal reduction to $\Gr(k,n)$. If $\dim \cV = k(n - \ell) - 1+i$ for some $0 \leq i \leq k(\ell - k),$ then $\cH$ equals the Chow--Lam discriminant $\mathcal{CL}^i_\cV$ of $\cV.$
\end{theorem}

The paper is organized as follows. In Section \ref{sec:conormal}, we review the conormal bundle and conormal varieties (Subsection \ref{subsec:conormal variety}), and introduce conormal rank (Subsection \ref{subsec:conormal rank}). We compute the conormal rank for various varieties, including Schubert varieties and general torus orbit closures. In Section \ref{sec:principal cl}, we recall the definition of the principal Chow--Lam locus from \cite{ChowLam} and give examples. In Section \ref{sec:cl discriminant} we introduce the Chow--Lam discriminant and prove that it is irreducible (Theorem \ref{thm:cldisc irreducible}). Section \ref{sec:rank conditions} is dedicated to the proof of Theorem \ref{thm:rankk}. In Section \ref{sec:recovery} we give a procedure for recovering, from each rank $k$ hypersurface $\cH$ in $\Gr(\ell, n),$ a subvariety $\cV$ of the Grassmannian $\Gr(k,n)$. We then prove Theorem \ref{thm:cl recovery}. Finally, Section \ref{sec:positroids} studies conormal rank of positroid varieties, and how to tell when a polynomial is a Chow--Lam form of a positroid variety.

	\section{Conormal varieties and conormal rank}\label{sec:conormal}

	This section introduces conormal bundles, conormal varieties, and conormal rank. The first two are standard definitions in the literature and the last is introduced in this paper.

	\subsection{Conormal varieties}\label{subsec:conormal variety}

	We review some relevant facts about conormal varieties and conormal bundles. Throughout, let $X$ be a smooth complex projective variety and let $Y$ be a closed subvariety of $X$ with regular locus $Y_{\rm reg}.$

	\begin{definition}
		The conormal bundle of $Y$ in $X$ is
		\[
		N^*_Y X
		\;:=\;
			\left\{
			(x,\xi) \in T^*X
			\;\middle|\;
			x \in Y_{\mathrm{reg}},\
			\xi|_{T_x Y_{\rm reg}} = 0
			\right\}
		\;\subset\; T^*X,
		\]
	where the bundle projection to $X$ is inherited from that of the cotangent bundle $T^*X.$ If the pair $(x, \xi)$ is in $N^*_YX,$ we call $\xi$ a \emph{covector} at $x$.
	\end{definition}

	The conormal bundle is a vector bundle over the regular locus of $Y$. For a non-smooth variety $Y$, the cotangent spaces will drop in dimension at singular points, so $N^*_Y X$ does not extend to a vector bundle over $Y$. The conormal variety is the correct replacement.

	\begin{definition}[Conormal variety]\label{def:conormal variety}
		The conormal variety of \(Y\) in \(X\) is
		\[
		\Con_Y (X)
		\;:=\;
		\overline{
			N^*_Y X
		}
		\;\subset\; T^*X,
		\]
		where the closure is taken in the Zariski topology. If $Y = \cup_i Y_i$ is reducible, we define its conormal variety to be the union $\cup_i \Con_{Y_i}X$ of the conormal varieties of the components. When $Y$ is smooth, the conormal variety equals the total space of the conormal bundle of $Y.$
	\end{definition}

	Conormal varieties are special subvarieties of $T^*X.$ There is a $\bC^*$ action on the cotangent bundle $T^*X$ which acts on each cotangent fiber by scalar multiplication. Conormal varieties are \emph{conic}, namely invariant under this action. They are \emph{Lagrangian}; that is, they are irreducible of dimension $\dim X$  and isotropic with respect to the natural symplectic form $\omega_X$ on the cotangent bundle of $X.$ In local coordinates $(x_1, \ldots, x_n, \xi_1, \ldots, \xi_n)$, this form looks like
	\[\omega_X = \sum_{i=1}^n {\rm d}x_i \wedge {\rm d}\xi_i.\]
	Moreover, these properties characterize conormal varieties. Indeed, every irreducible conic Lagrangian
subvariety of $T^*X$ is the conormal variety of its
projection to $X$; see e.g. \cite[Lemma 4.2]{behrend}.

	\begin{remark}
	With our conventions conormal varieties are quasi-affine, but not projective. When $X$ is projective space $\bP^n,$ some authors define the conormal variety as the closure
	\begin{align*}
		\overline{\{(y, H)  \ | \ y \in Y_{\text{reg}}, \ T_yY \, \subset \, H\}} \ \subset \ \bP^n \times (\bP^n)^\vee.
	\end{align*}
	This yields a projective variety of dimension $n-1$, which is isomorphic to the projectivization of the twisted conormal bundle $N^*_{Y}\bP^n (1).$ Indeed, the twisted Euler sequence and conormal sequence give an injection $N^*_{Y }\bP^n(1) \to \cO_Y^{n + 1}.$ Projectivizing these bundles then determines an embedding $\bP(N^*_{Y}\bP^n (1)) \to Y \times (\bP^n)^\vee.$ This convention is used for instance in \cite{kohn}.
	\end{remark}

	The following lemma says that the total space of the conormal bundle is dense in the conormal variety, and the Zariski closure adds nothing over the regular locus. It is used in a minor way in Section \ref{sec:cl discriminant}, after Definition \ref{def:cl discriminant}.

	\begin{lemma}\label{lem:con vs bundle}
	Let $Y$ be a subvariety of $X$, write $Y_{\rm sing} := Y \setminus Y_{\rm reg}$
	for its singular locus, and let $\pi : T^*X \to X$ denote the bundle map. Then
	\[
	\Con_Y(X) \, \cap \, \pi^{-1}\bigl(X \setminus Y_{\rm sing}\bigr)
	\;=\;
	N^*_Y X.
	\]
	In particular, $N^*_YX$ is a dense open subset of $\Con_Y(X),$ and
	$\Con_Y(X) \setminus N^*_YX$ lies over $Y_{\rm sing}.$
\end{lemma}

\begin{proof}
	Write $U := \pi^{-1}(X \setminus Y_{\rm sing}),$ an open subset of $T^*X$
	containing $N^*_YX.$ Inside $U$ the subset $N^*_YX$ is closed, since it is cut out by the conditions $x \in Y$ and $\xi|_{T_xY} = 0$. Note that for any subset $S$
	which is closed in an open set $U$ we have $\overline{S} \cap U = S.$ Applying this to
	$S = N^*_YX$ gives the
	displayed equality and subsequent statements.
\end{proof}

	\subsection{Conormal rank}\label{subsec:conormal rank}
    In this section we define the conormal rank of a subvariety of the Grassmannian and establish its basic properties. Each fiber of the cotangent bundle $T^*\Gr(k,n)$ has a concrete description as a space of homomorphisms. Let $\cS$ denote the universal sub-bundle and $\cQ$ denote the universal quotient bundle on $\Gr(k,n),$ whose fibers at a point $L$ are given by $L$ and $\bC^n / L,$ respectively. Then $T\Gr(k,n) \cong \sHom(\cS, \cQ)$; see e.g. \cite[B.5.8]{fulton}. Thus the tangent space at a point $L$ can be naturally identified with the space $\Hom(L, \bC^n/L)$ of homomorphisms. 

	Consider now the trace inner product defined by
	\begin{equation*}
	\begin{split}
	\Hom(L, \bC^n/L) \times \Hom(\bC^n/L, L) & \to \bC \\
	(u, v) & \mapsto \mathrm{tr}(u \circ v) = \mathrm{tr}(v \circ u).
	\end{split}
	\end{equation*}
	This inner product realizes the isomorphism $(\cS^* \otimes \cQ)^* \cong  \cQ^* \otimes \cS.$ Concretely, it identifies $T^*_L \Gr(k,n) = \Hom(L, \bC^n/L)^*$ with $\Hom(\bC^n/L, L).$ This allows us to attach an extra statistic to each conormal vector: its rank as a homomorphism.

    \begin{definition}
    For a nonnegative integer $r$, we define the rank $r$ locus to be
    \[
D_r := \{(L,\eta) : \operatorname{rk}(\eta) \leq r\}
\subset T^*\Gr(k,n).
\]
    The \emph{rank stratification} of the cotangent bundle $T^*\Gr(k,n)$ is
    \[D_0 \subset D_1 \subset \ldots \subset D_{\min (k, n - k)} = T^*\Gr(k,n).\]
    \end{definition}

   The definition of conormal rank from the introduction is more naturally written using the language of rank stratifications.
    \begin{definition}\label{def:conormal rank}
    Let $\cY \subset \Gr(k,n)$ be an irreducible variety. The \emph{conormal rank} of $\cY$ is the smallest integer $r$ such that $N^*_{\cY}\Gr(k,n)$ is contained in $D_r.$
    \end{definition}

    In particular, $D_r$ is closed, so it contains $\Con_{\cY}\Gr(k,n)$ whenever it contains $N^*_\cY \Gr(k,n).$ We show that Definition \ref{def:conormal rank} is equivalent to imposing rank conditions on a general covector.

    \begin{lemma}\label{lem:crk open}
    Let $\cY \subset \Gr(k,n)$ be an irreducible variety. The conormal rank of $\cY$ is the rank of $\eta$ for a general covector $\eta \in \Con_\cY\Gr(k,n).$ For a general point $L \in \cY$, the conormal rank of $\cY$ equals the maximum rank of the covectors in the fiber $N^*_\cY\Gr(k,n)_L.$
    \end{lemma}
    \begin{proof}
    The dropping of rank is a closed condition on $\Con_\cY\Gr(k,n),$ so the maximum rank of a covector equals the generic rank. Since $N^*_\cY\Gr(k,n)$ is a vector bundle on $\cY_{\rm{reg}}$, the set of covectors $\eta$ whose rank attains the maximum projects to an open set in $\cY.$ This proves the second statement.
    \end{proof}

	\begin{example}
		Suppose that $k=1.$ Every proper irreducible subvariety $\cV$ of projective space has conormal rank one. Indeed, the rank of each nonzero conormal vector is one.
	\end{example}

	\begin{example}[Coisotropic hypersurfaces in $\Gr(2,4)$]
		A \emph{coisotropic hypersurface} is a hypersurface of conormal rank one. The name ``coisotropic" is because the cotangent vector to such a hypersurface at regular point is isotropic with respect to the determinant form.
		
		For a concrete example, consider the $\Gr(2,4)$ case. The tangent space at a point $L$ is $V := \Hom(L, \bC^4 / L).$ Any vector $w \in V^*$ corresponds to a $ 2 \times 2$ matrix $M_w$. Then $\gen{w, w} = \det (M_w),$ which equals zero whenever $M_w$ has rank one.

		By \cite[Theorem 4.3.14]{gkz}, every coisotropic hypersurface in $\Gr(2,4)$ is either the \emph{Chow locus} of a curve in $\bP^3,$ parameterizing projective lines which meet the curve, or the \emph{Hurwitz locus} of a surface in $\bP^3,$ parameterizing projective lines tangent to the surface.
	\end{example}

	\begin{example}[Schubert divisors]
		Every Schubert divisor is the Chow hypersurface of a linear space, hence conormal rank one by \cite[Theorem 4.3.14]{gkz}. We may also see this concretely as follows. Let $Q \subset \bC^n$ have codimension $k$ and let  $\Omega(Q)\subset \Gr(k,n)$ be the Schubert divisor of $k$-dimensional linear spaces in $\bC^n$ meeting $Q$ in a dimension at least one. It is a standard fact that
		\[T_L \Omega(Q) = \{\phi: L \to \bC^n / L \, : \, \phi(L \cap Q) \subset Q / Q \cap L\}.\]
		Thus the conormal vector $\phi: \bC^n / L \to L$ is determined up to scalar by the conditions
		\[\im \phi \subset Q \cap L, \quad \phi((Q + L) / L) = 0.\]
        For a general point $L,$ the dimension of $L \cap Q$ is one, so the conormal vector has rank $1.$
	\end{example}

	The conormal rank can be computed more generally for any Schubert variety. Fix positive integers $k \leq n$ and a partition
$\lambda = (\lambda_1, \ldots,\lambda_k)$ which fits in a
$k \times (n-k)$ box. (Here we identify $\lambda$ with its Ferrers diagram, which has $\lambda_i$ boxes in row $i$). For $i \in [1,n],$ let $E_i$ be the subspace of $\bC^n$ spanned by the
last $i$ coordinate vectors. Fix the reference flag $E_1 \subset \cdots \subset E_n$. We define the \emph{Schubert variety}
$\Omega_\lambda$ to be
\begin{equation}\label{eq:schubert def}
\Omega_\lambda
=
\left\{
L \in \operatorname{Gr}(k,n)
:
\dim L \cap E_{n-k+i-\lambda_i} \geq i
\text{ for } i \in [1,k]
\right\}.
\end{equation}
	\begin{proposition}\label{prop:schubert rank}
		Fix $n,k$ with $k \leq n$, and a partition
$\lambda \subset k \times (n-k).$ The conormal rank of the Schubert variety $\Omega_\lambda$ is \[
\operatorname{crk}(\Omega_\lambda)
=
\max\left\{
\# S :
S\subseteq \lambda,\
\text{no two boxes of }S\text{ lie in the same row or column}
\right\}.
\]
	\end{proposition}
	\begin{proof}
	By Lemma \ref{lem:crk open}, we may compute the conormal rank on open sets. Let $\Omega^\circ_\lambda$ be the open Schubert cell where the rank inequalities in \eqref{eq:schubert def} are equalities. A point in $\Omega^\circ_\lambda$ can be represented uniquely by taking the reduced row echelon form of any matrix representative $A.$ The pivots of any such point will lie in the same set $(i_1, \ldots, i_k),$ allowing us to work in the affine chart where the Pl\"ucker coordinate $p_{i_1 \cdots i_k}$ is nonzero. The pattern of zero entries in the non-pivot columns is precisely the Ferrers diagram of $\lambda$ (in the French notation, with the largest part at the bottom). Thus in this open cell, the covectors are given by $k \times (n-k)$ matrices with support in $\lambda.$ The maximum rank of such a matrix equals the maximum in the statement.
	\end{proof}

	We may also compute the conormal rank of torus orbit closures via explicit parameterization. The torus $(\bC^*)^n$ acts on the Grassmannian $\Gr(k,n)$ by scaling the columns of a $k \times n$ matrix representative of each point, or equivalently on the Pl\"ucker coordinates by
	\[t \cdot p_{i_1 \ldots i_k} := t_{i_1} \cdots t_{i_k}p_{i_1 \ldots i_k}.\]
	For a point $L \in \Gr(k,n)$ let $\cV_L$ denote the torus orbit closure $\overline{(\bC^*)^n \cdot L}.$ The torus has a stabilizer $(t, \, \ldots, \, t),$ so for a general point $L$ the variety $\cV_L$ has dimension $n-1.$

	\begin{proposition}\label{prop:torus orbit rank}
		Fix positive integers $k, n$ with $k \geq 2$ and $2k < n.$ Choose a point $L \in \Gr(k,n)$ whose Pl\"ucker coordinates are nonzero. Then $\crk \cV_L = k.$
	\end{proposition}

	\begin{proof}
	A general point $L$ in $\Gr(k,n)$ may be represented as the rowspan of $[I_k : A]$ where $A$ is a $k \times (n-k)$ matrix. Then the torus action in this chart is given by scaling the rows and columns of $A.$ Parametrically, the torus action on $A$ is given by 
	\[(s_1, \ldots, s_k, t_1, \ldots, t_{n-k}) \cdot A = \begin{bmatrix}
	s_1t_1a_{1,1} & \hdots & s_1t_{n-k}a_{1,n-k} \\
	\vdots & \ddots & \\
	s_kt_1a_{k,1} & & s_kt_{n-k}a_{k,n-k}
	\end{bmatrix}.\]
	A matrix $C = (c_{ij})_{i \in [k], \, j \in [n-k]}$ is a covector at the point $(1, 1) \cdot A$ whenever it pairs to zero with each tangent vector $\frac{\partial}{\partial s_i}, \frac{\partial}{\partial t_j}$ at $A,$ or equivalently
	\begin{equation}\label{eq:torus covector}
	\sum_{i = 1}^k a_{ij} c_{ij} = 0 \quad \forall j \in [n-k], \qquad \sum_{j = 1}^{n-k} a_{ij} c_{ij} = 0 \quad \forall i \in [k].
	\end{equation}
	All minors of $A$ are maximal minors of $[I_k : A],$ hence nonzero by assumption. Define the $k \times (n-k)$ matrix $C$ by taking its first $k$ columns to be
$$\begin{pmatrix}
\frac1{a_{11}} & 0 & \cdots & 0 & \frac{1}{a_{1k}}\\[2pt]
0 & \frac1{a_{22}} & & & 0\\
\vdots & &\ddots & &\vdots\\[2pt]
0 & & &\frac1{a_{k-1,k-1}} & 0\\[2pt]
-\frac1{a_{k1}} & -\frac1{a_{k2}} &\cdots & -\frac{1}{a_{k,k-1}} & -\frac1{a_{kk}}
\end{pmatrix},$$
the $(k+1)$st column to be $\big(-\tfrac{2}{a_{1,k+1}},\,-\tfrac1{a_{2,k+1}},\,\dots,\,-\tfrac1{a_{k-1,k+1}},\,\tfrac{k}{a_{k,k+1}}\big)^{T}$, and remaining entries zero. Then $C$ satisfies \ref{eq:torus covector} and has rank $k,$ so $\cV_L$ has rank $k.$
	\end{proof}

	It would be interesting to understand ranks of torus orbit closures of non-general points.



	\section{Principal Chow--Lam loci}\label{sec:principal cl}
	In this section we recall the definition of a principal Chow--Lam locus of a subvariety $\cV$ in a Grassmannian, originally given in \cite{ChowLam}. It captures tangency between $\cV$ and certain subGrassmannians. We first introduce some notation for the subGrassmannians.

	\begin{enumerate}
		\item[(i)] Fix an $\ell$-dimensional linear space $Q \subset \bC^n$ and an integer $k \leq \ell < n$.
		The set
		\[
		\Gr(k,Q) \; :=\; \{\, V \in \Gr(k,n) \mid V \subseteq Q \,\}
		\]
		of $k$-spaces contained in $Q$ is a Schubert variety, linearly isomorphic to
		\(
		\Gr(k, \ell).
		\)

		\item[(ii)] Fix a $k$-dimensional linear space $L \subset \bC^n$, and an integer $k \leq \ell < n$.
		The set
		\[
		\Gr(L,\ell) \; :=\; \{\, V \in \Gr(\ell,n) \mid L \subseteq V \,\}
		\]
		of $\ell$-spaces containing $L$ is a Schubert variety, linearly isomorphic to
		\(
		\Gr(\ell - k, n - k).
		\)
	\end{enumerate}

The principal Chow--Lam locus takes as input a fixed subvariety $\cV$ of a Grassmannian and then varies a linear space $Q$ of fixed dimension, asking when $\Gr(k,Q)$ is tangent to $\cV.$ We recall the definition from the introduction, with the precise indexing, as follows.
\begin{repdefinition}{def:principal cl}
Let $\cV\subset\Gr(k,n)$ be an irreducible variety of dimension
$k(n-\ell)-1+i$, where $k< \ell < n$ and $0\leq i\leq k(\ell - k).$ Define the incidence variety
\[
\Phi_{\cV}^i
=
\overline{
\left\{
(L,Q):
L\in\cV_{\rm reg}
\text{ and }
\Gr(k,Q)\cap\cV
\text{ is not transverse at }L
\right\}} \, \subset \, \Fl(k,\ell,n).
\]
The $i$th \emph{principal Chow--Lam locus} $\mathcal{PCL}_{\cV}^i$ is the projection of $\Phi_{\cV}^i$ to $\Gr(\ell,n).$ If $\mathcal{PCL}_{\cV}^i$ is a hypersurface, its equation is
called the $i$th \emph{principal Chow--Lam form} and denoted ${\rm PCL}^i_\cV$.
\end{repdefinition}

\begin{remark}\label{rmk:principal-vs-higher}
Principal Chow--Lam forms were called higher Chow--Lam forms and denoted ${\mathcal CL}^i_\cV$ in \cite{ChowLam}. We prefer to reserve the name Chow--Lam discriminant and notation $\mathcal{CL}^i_\cV$ for the irreducible polynomial we define in the following section, which will be a factor of the principal Chow--Lam form. For the ordinary Chow--Lam case, the two notions coincide when neither is $1$; see Remark \ref{rem:cl conventions}.
\end{remark}

The case $i = 0, \, k=1$ recovers the classical Chow form, introduced by Chow and van der Waerden \cite{CvW} based on earlier work of Cayley for curves in projective three-space \cite{cayley}. More recently, they have appeared in the construction of Chow quotients by Kapranov \cite{kapranov}.

More generally, the case $i = 0$ is called the \emph{Chow--Lam locus}. It describes subGrassmannians of the form $\Gr(k,Q)$ whose dimension equals $\codim  \cV - 1$ and which meet the variety $\cV.$ For concreteness we include the definition.
\begin{definition}
		Consider  $\mathcal{V} \subset \Gr(k,n)$,
		where $\dim (\mathcal{V}) = k(n-\ell)-1$ for some $k < \ell < n$. The \emph{Chow--Lam locus} is
		\[\mathcal{CL}_\cV := \overline{\{Q \ : \ \Gr(k,Q) \cap \cV \neq \varnothing\}} \ \subset \ \Gr(\ell,n).\]
		The codimension of $\mathcal{CL}_\mathcal{V}$
		in $\Gr(\ell,n)$ is expected to be one.
		If it is one, then
		$\mathcal{CL}_\mathcal{V}$    is defined by a homogeneous
		polynomial  in Pl\"ucker coordinates, which is unique modulo Pl\"ucker relations.
		This polynomial is denoted
		${\rm CL}_\mathcal{V}$ and called the
		{\em Chow--Lam form} of $\mathcal{V}$.
	\end{definition}

The case $k=1$ recovers the \emph{associated hypersurfaces} introduced by Gelfand, Kapranov, and Zelevinsky \cite{gkz}. 
One major difference between the $k=1$ case and the more general Grassmannian case is that for $k>1$, there is a dimension restriction on the input variety~$\cV$. This restriction is designed so that we expect principal Chow--Lam loci to be hypersurfaces. For example, a subvariety $\cV \subset \Gr(2,4)$ will have a principal Chow--Lam form for $i = 0$ if it is a curve, for $i=1$ if it is a surface, and for $i = 2$ if it is a threefold. For each of these cases, $\ell = 3$ and the principal Chow--Lam locus lies in $\Gr(3,4).$

	\begin{example}[Chow--Lam form of a curve]\label{eg:clcurve}
		Let $\cV$ be a curve in $\Gr(2,4).$ Define
		\[\Phi_\cV := \{(L, P) \ : \ L \subset P \text{ and $L$ is in } \cV_{\rm reg}\} \ \subset\  \cV \times \Gr(3,4).\]
		Then the projection to $\Gr(3,4) \cong (\bP^3)$ is a surface, and its defining equation is the \emph{Chow--Lam form} of the curve. Indeed, let $X_\cV$ be the surface in $\bP^3$ ruled by the curve $\cV.$ Then the Chow--Lam locus equals the dual of $X_\cV$ whenever the latter is a hypersurface.
	\end{example}

	\begin{example}[Hurwitz--Lam form of a surface]
		Let $\cV$ be a surface in $\Gr(2,4),$ also called a \emph{line congruence}. Each projective plane $P$ in $\bP^3$ determines a Schubert variety $\Gr(2,P)$ of projective lines contained in $P.$ We define
		\[\Phi := \{(L,P) \ : \ \Gr(2,P) \text{ tangent to }\cV \text{ at $L$ in } \cV_{\rm reg}\} \ \subset \ \Gr(2,4) \times \Gr(3,4).\]
		The projection to $\Gr(3,4)$ is called the \emph{Hurwitz--Lam locus}. It parameterizes planes containing fewer than expected lines from the congruence.
	\end{example}

\begin{example}[Chow--Lam form of a torus orbit closure]
		Let $L$ be a general point of $\Gr(2,6)$, with Pl\"ucker coordinates $q_{ij}.$ In the case $k=2$ and $\ell = 3$, its torus orbit closure $\cV_L$ is a $5$-dimensional subvariety of $\Gr(2,6).$ The corresponding Chow--Lam locus lives in $\Gr(3,6)$ and parameterizes planes in $\bP^5$ containing a line $L'$, such that $L'$ is a point of $\cV_L.$ It is a hypersurface whose defining equation is
		\[\begin{matrix} {\rm CL}_\cV = (q_{12}q_{34}q_{56} \,+\, q_{14}q_{25}q_{36})\, p_{123} p_{456}
			\,-\,  q_{13} q_{25} q_{46} \, p_{124} p_{356} \\
			\,+\,  q_{12} q_{35} q_{46} \, p_{134} p_{256}
			\,-\,  q_{12} q_{34} q_{56} \, p_{135} p_{246}
			\,+\,  q_{13} q_{24} q_{56} \, p_{125} p_{346}.
		\end{matrix} \]
		One may check using a computer algebra system that in any affine chart, all of the conormal vectors have rank at most two. The Chow-Lam form here was computed in \cite[Example 4.3]{ChowLam}; more generally, the Chow--Lam form of a torus orbit closure is described in \cite{SegreDet}.
	\end{example}
It was shown in \cite{ChowLam} that the Chow--Lam form is irreducible. However, principal Chow--Lam forms are sometimes reducible for $i$ greater than zero.

 \begin{example}[A Schubert linear section] \label{eg:schubert reducible}
		Let $\cV$ be the positroid variety in $\Gr(2,6)$ given by $V(q_{12},q_{34},q_{56}).$ This is the transverse intersection of three Schubert divisors, given by lines intersecting the three-spaces $\overline{e_1e_2e_3e_4}, \, \overline{e_3e_4e_5e_6},$ and $\overline{e_1e_2e_5e_6}.$ The principal Chow--Lam locus $\mathcal{PCL}^2_\cV$ parameterizes projective three-spaces $Q$ in $\bP^5$ such that the curve $\Gr(2,Q) \cap \cV$ is singular. Indeed, it is given by $p_{1234}p_{3456}p_{1256} = 0$  \cite[Example 5.6]{ChowLam}.
 \end{example}

The following gives an example where the Hurwitz--Lam form is reducible.

 \begin{example}[The join of two conics]\label{eg:conics}
Let $A, B$ be three-dimensional complex vector spaces. Let $Q_1$ be a smooth conic in the projective plane $\bP A$ and let $Q_2$ be a smooth conic in $\bP B.$ We define the variety $\cV$ to be
\[
\cV
:=
\left\{
\overline{pq} \, \in\Gr(2,A\oplus B)
:
p\in Q_1,\ q\in Q_2
\right\}.
\]
Since $A\cap B=0$, the points $p$ and $q$ are recovered from
$L=\overline{pq} $ as $\bP(L\cap A)$ and $\bP(L\cap B)$.
Hence
\[
\cV\simeq Q_1\times Q_2\simeq\bP^1\times\bP^1.
\]

Choose coordinates so that
\[
Q_1=[s^2:st:t^2]\subset\bP A,
\qquad
Q_2=[u^2:uv:v^2]\subset\bP B.
\]
For a hyperplane
\[
H_y=V(y_1x_1+\cdots+y_6x_6)\subset\bP(A\oplus B),
\]
the condition $\overline{pq} \, \subset H_y$ becomes
\[
y_1s^2+y_2st+y_3t^2=0,
\qquad
y_4u^2+y_5uv+y_6v^2=0.
\]
The intersection of $H_y$ with $\cV$ fails to be transverse precisely when one of the
two binary quadratics has a double root. Thus the Hurwitz--Lam
locus has two irreducible components, and its defining equation is
\[
\mathrm{HL}_{\cV}
=
\left(y_2^2-4y_1y_3\right)
\left(y_5^2-4y_4y_6\right).
\]
The symmetry of swapping the two factors reflects the symmetry of swapping the two conics.
 \end{example}

\section{The Chow--Lam discriminant}\label{sec:cl discriminant}
A central philosophy in the theory of discriminants is that they should be made to be irreducible.  One restricts the incidence correspondence to a ``nice'' locus in the base variety $\cV$, then takes Zariski closure. This procedure removes components corresponding to less meaningful tangencies.

This philosophy appears in many contexts. For example, in the definition of a dual variety of a projective variety $\cV$, one restricts the incidence correspondence to the regular locus of $\cV$ and then takes Zariski closure. This avoids picking up extra components from hyperplanes passing through the singular locus of $\cV$.

Another example comes from the $A$-discriminant in toric geometry. Let $A \subset \bZ^m$ be a set of integer vectors generating $\bZ^m.$ The $A$-discriminant is the defining equation of the closure of the set
\[
\nabla_A := \left\{
f \in \mathbb{C}^A
\;\middle|\;
V\left(
f,
\frac{\partial f}{\partial x_1},
\ldots,
\frac{\partial f}{\partial x_{m}}
\right)
\neq \varnothing
\text{ in } (\mathbb{C}^*)^{m}
\right\}.
\]
Note that one allows repeated roots only in the dense torus, and then takes Zariski closure. There is another polynomial associated to $A$, called the principal $A$-determinant, which is easier to compute and contains the $A$-discriminant as an irreducible factor. For a thorough explanation of these, see \cite[Chapters 9--10]{gkz}. It is in this spirit that we chose the name ``principal Chow--Lam form'' for the polynomial in the previous section, which is reducible in general (cf. Example \ref{eg:schubert reducible}). 

In this section we introduce the Chow--Lam discriminant, which is an irreducible factor of the principal Chow--Lam form that records the important tangencies. We prove irreducibility in Theorem~\ref{thm:cldisc irreducible}. The lemmas required for that proof are reused in Sections \ref{sec:rank conditions} and \ref{sec:recovery}.

\vspace{0.3cm}

Fix integers $k < \ell < n$. Let $F := \mathrm{Fl}(k, \ell, n)$ be the partial flag variety parameterizing flags $V_k \subset V_{\ell} \subset \bC^n$ of  vector subspaces, where $V_k, V_\ell$ have dimensions $k, \ell,$ respectively. Let \[X:= \Gr(k,n) \times \Gr(\ell, n)\] be the product of the relevant Grassmannians, and let $p_1, p_2$ be the projections to each factor, as in Figure \ref{fig:flagprojections}.

\begin{figure}[!h]
		\begin{center}
\begin{minipage}{0.49\textwidth}
    \centering
    \begin{tikzcd}[column sep=tiny]
        & F = \Fl(k,\ell,n) \subset X
        \arrow{dl}[swap, dash]{p_1}
        \arrow{dr}[dash]{p_2} & \\
        \Gr(k,n) & & \Gr(\ell,n)
    \end{tikzcd}
\end{minipage}%
\hfill
\begin{minipage}{0.49\textwidth}
    \centering
    \begin{tikzcd}[column sep=tiny]
        & N^*_{F}X
        \arrow{dl}[swap, dash]{p_1}
        \arrow{dr}[dash]{p_2} & \\
        T^*\Gr(k,n) & & T^*\Gr(\ell,n)
    \end{tikzcd}
\end{minipage}
\end{center}
		\caption{Maps from the partial flag variety and its conormal bundle}\label{fig:flagprojections}
	\end{figure}

Let $N^*_FX$ denote the conormal bundle of the flag variety in $X$. Then we have the maps on the right side of Figure \ref{fig:flagprojections}, where we re-use the notation for the induced projections to the cotangent bundles $T^*\Gr(k,n)$ and $T^*\Gr(\ell, n)$. We define $D^\circ_r := D_r \setminus D_{r-1}$ to be the (open) conormal rank $r$ locus in the cotangent bundle $T^*\Gr(k,n).$

\begin{notation}
For $\cV$ a subvariety of $\Gr(k,n),$ it is convenient to fix the notation 
\[\Con(\cV) := \Con_\cV \Gr(k,n),\]
when the ambient Grassmannian is clear from context. We will also use the notation 
\[
\Con(\cV)^\circ_k
:=
N^*_{\cV}\Gr(k,n) \cap D_k^\circ.
\]
This is the subset of the conormal bundle consisting of covectors of rank exactly $k.$ 
\end{notation}

Note that since the dropping of rank is a closed condition, and there are no covectors of rank more than $k$, the set $\Con(\cY)^\circ_k$ is either empty or Zariski open in $\Con(\cY).$

\begin{definition}[Chow--Lam discriminant]\label{def:cl discriminant}
	Let $\cV \subset \Gr(k,n)$ be an irreducible variety of dimension $k(n-\ell)-1 + i,$ where $k < \ell \leq n-k$ and $0 \leq i \leq k(\ell - k).$ Define the incidence variety
	\[\Psi^i_\cV = \overline{\{(L, \phi) \times (Q, \eta) \in N^*_FX \ : \ (L, \phi) \in \Con(\cV)^\circ_k\}}.\]
	If the projection to $\Gr(\ell, n)$ is a hypersurface, we call its defining equation the \emph{Chow--Lam discriminant} of $\cV$ and $i$ and denote it $\mathrm{CL}^i_\cV.$ Else, we set  $\mathrm{CL}^i_\cV := 1.$
\end{definition}

Note the upper bound $\ell \leq n-k$ in Definition \ref{def:cl discriminant}, compared to $\ell < n$ in Definition \ref{def:principal cl}. In particular the bound implies that $k < n - k.$ When these bounds are violated, the incidence variety is empty and the Chow--Lam discriminant is $1.$ However, this does not impose a large restriction on $k$ and $n$; if $k$ is greater than $n-k,$ then one may apply the map which takes each subspace $L$ to its orthogonal complement $\L^\perp \subset \bC^n,$ and work with $\Gr(n-k,n)$ instead. Thus the only truly problematic choice is $n = 2k.$

\begin{remark}
The incidence variety $\Psi^i_\cV$ is unchanged if one replaces $\Con(\cV)^\circ_k = N^*_\cV \Gr(k,n) \cap D^\circ_k$ with $\Con(\cV) \cap D^\circ_k.$ Indeed, by Lemma \ref{lem:con vs bundle}, we have that $N^*_\cV \Gr(k,n)$ is dense and open in $\Con(\cV).$ Thus their intersections with $D^\circ_k$ are either both empty or both dense. In the latter case, since $p_1$ restricts to a fiber bundle, the pre-images in $\Psi^i_\cV$ are dense.
\end{remark}

\begin{remark}\label{rem:cl conventions}
In the $i=0$ case the Chow--Lam form is irreducible \cite[Lemma 3.2]{ChowLam}, and coincides with the Chow--Lam discriminant when the latter is not $1$; see Theorem \ref{thm:cldisc irreducible}. The sources \cite{ChowLam,PrattRanestad,SegreDet} use the notation ${\mathcal CL}_\cV$ for the Chow-Lam form, which agrees with ${\mathcal CL}^0_\cV$ whenever the latter is defined, so there is no major notational clash.
\end{remark}

 To prove that this discriminant is irreducible, we will need to understand the conormal bundle of the flag variety concretely. Let $\cS_k$ denote the pullback to $X$ of the universal subbundle on
$\Gr(k,n)$, and let $\cQ_\ell$ denote the pullback to $X$ of the
universal quotient bundle on $\Gr(\ell,n)$. Set
$
E:=\sHom(\cS_k,\cQ_\ell).$
The fibers of $E$ and $E^*$ are given concretely by
\[
E_{(L,Q)}=\Hom(L,\bC^n/Q), \quad E^*_{(L,Q)}=\Hom(\bC^n/Q, L).
\]
\begin{lemma}\label{lem:conormal}
There is a natural isomorphism
\[
N_F^*X\cong E^*|_F.
\]
\end{lemma}

\begin{proof}
Let $s$ be the global section of $E$ induced by
\[
\cS_k \hookrightarrow \bC^n \otimes \cO_X
\longrightarrow \cQ_\ell.
\]

At each fiber $(L,Q)$ the section $s$ is given by the composition
$L\hookrightarrow\bC^n\twoheadrightarrow\bC^n/Q$ in $\Hom(L, \bC^n/Q).$
Thus the zero locus of $s$ is precisely $F$.
Let $p:\bC^n/L\twoheadrightarrow\bC^n/Q$ denote the induced projection map. Under the identifications
$T_{(L,Q)}X=\Hom(L,\bC^n/L)\oplus\Hom(Q,\bC^n/Q)$ and
$E_{(L,Q)}=\Hom(L,\bC^n/Q)$ we have
\[ds(\alpha,\beta)=p\circ\alpha-\beta|_L \quad \in E_{(L, Q)}.\] The differential $ds$ is surjective at every point of $F$ (taking $\beta = 0$ and varying $\alpha$ gives all points in the target). Thus $s$ is a transverse section of $E$ and $F$ is its smooth zero scheme, and $ds$
fits in a short exact sequence
\[
0\longrightarrow T_F\longrightarrow T_X|_F\xrightarrow{ds} E|_F\longrightarrow 0 .
\]
Dualizing and comparing with the conormal sequence yields $N_F^*X\cong E^*|_F$.
\end{proof}

The two projections in the conormal correspondence of $
u:\bC^n/Q\longrightarrow L$ are therefore:
\begin{equation}\label{eq:conormal projections}
p_1(u):
\bC^n/L\twoheadrightarrow\bC^n/Q
\xrightarrow{u}L, \quad \quad p_2(u):\bC^n/Q\xrightarrow{u}L\hookrightarrow Q.
\end{equation}

\begin{theorem}\label{thm:cldisc irreducible}
Let $\cV \subset \Gr(k,n)$ be an irreducible variety of dimension $k(n-\ell)-1 + i,$ where $k < \ell \leq n-k$ and $0 \leq i \leq k(\ell - k).$ The Chow--Lam discriminant $\mathrm{CL}^i_\cV$ is an irreducible factor of the principal Chow--Lam form $\mathrm{PCL}^i_\cV$.
\end{theorem}
\begin{proof}
Suppose that $\Con(\cV)^\circ_k$ is empty. Then the incidence variety is empty and $\CL^i_\cV = 1,$ so the statement is vacuously true.

Now suppose that $\Con(\cV)^\circ_k$ is nonempty. The bundle $N^*_\cV\Gr(k,n)$  is a vector bundle over the irreducible variety $\cV_{\rm reg}.$ Thus it is irreducible. Furthermore, the restriction of $p_1$ to $p_1^{-1}(D^\circ_k)$ is a fiber bundle with irreducible fiber. Indeed, its fiber above $(L, \phi)$ consists of pairs $(Q, u)$ where $\phi: \bC^n / L \to L$  factors through $u$ as in \eqref{eq:conormal projections}.
In symbols, we have
\begin{equation}\label{eq:cl fiber}
p_1^{-1}(L, \phi) = \{(Q, u) \, : \, L \subset Q, \ Q /L \subset \ker \phi\}.
\end{equation}
But since $\phi$ has rank $k,$ the fibers are all isomorphic to $\Gr(\ell-k, n-2k)$. Thus the incidence correspondence restricts to a fiber bundle above $\Con(\cV)^\circ_k.$ This restriction is irreducible and so is the Zariski closure $\Psi^i_\cV$. Therefore the projection to $\Gr(\ell, n)$ is irreducible.

To show $\CL^i_\cV$ divides ${\rm PCL}^i_\cV$, we note that $\phi$ annihilates $T_L\Gr(k, Q),$ since $\phi(Q/L) = 0.$ Thus $\Gr(k,Q)$ and $\cV$ intersect non-transversely, and the Chow--Lam discriminantal locus, if it is a hypersurface, is an irreducible component of the principal Chow--Lam locus.
\end{proof}

The Chow--Lam discriminantal locus given by projecting $\Psi^i_\cV$ is not always a hypersurface, and the Chow--Lam discriminant may equal $1,$ as in the following example.
\begin{example}\label{eg:cl curve join}
Let $\cV$ be the intersection of three Schubert divisors as in Example \ref{eg:schubert reducible}. Consider the change of coordinates $(e_1, e_2, e_3, e_4, e_5, e_6) \mapsto (e_3, e_4, e_5, e_6, e_1, e_2)$ on $\bC^6.$ The variety $\cV$ is invariant under the change of coordinates induced on the Grassmannian $\Gr(2, 6).$ Thus the Chow--Lam discriminant $\CL^2_\cV$ is invariant under the induced change of coordinates on $\Gr(4,6).$ But this permutes the irreducible (non-scalar) factors in the principal Chow--Lam form ${\rm PCL}^2_\cV$, so the Chow--Lam discriminant cannot equal any of the factors.
\end{example}

The following question would be interesting to address:
\begin{question}
For which varieties $\cV$ in $\Gr(k,n)$ is the Chow--Lam discriminant not $1$?
\end{question}
This is a large generalization of the classical question: for which projective varieties is the dual variety a hypersurface? There are criteria available for the $k=1$ case in terms of polar degrees of the variety $\cV,$ which are given by the class of the conormal variety; see e.g. \cite[Section 4]{kohn}. There is also a criterion for the $i=0$ case in terms of the cohomology class of $\cV$ \cite[Theorem 3.5]{ChowLam}. It would be interesting (though outside the scope of the present paper) to establish such criteria for subvarieties of general Grassmannians.

	\section{Rank conditions}\label{sec:rank conditions}

In this section we study conormal ranks of the hypersurfaces defined by principal Chow--Lam forms and Chow--Lam discriminants. Our main result is Theorem \ref{thm:rankk}, which states that any principal Chow--Lam locus of a
subvariety of $\Gr(k,n)$, when it is a hypersurface, has conormal rank bounded above by $k$. Furthermore, the Chow--Lam discriminant will define a hypersurface of conormal rank exactly $k$ (Theorem \ref{thm:cl rank}).

As in the previous section, let $p_1, p_2$ denote the projections from the bundle $E^*|_F \cong N^*_F X$ to the factors $T^*\Gr(k,n)$ and $T^*\Gr(\ell, n).$ We will need the following concrete description of the projections of the fibers.
\begin{lemma}\label{lem:eimage}
Let $F$ denote the flag variety $\mathrm{Fl}(k, \ell, n)$. For a flag $(L\subset Q) \in F$, we have
\[
p_1(E^*_{(L,Q)})
=
\left\{
\gamma\in\Hom(\bC^n/L,L)
\ :\
\gamma(Q/L)=0
\right\} = \Ann\bigl(T_L\Gr(k,Q)\bigr),
\]
and
\[
p_2(E^*_{(L,Q)})
=
\left\{
\eta\in\Hom(\bC^n/Q,Q)
\ :\
\operatorname{im}\eta\subseteq L
\right\} = \Ann\bigl(T_Q\Gr(L,\ell)\bigr),\]
where the annihilator is with respect to the trace pairing.
\end{lemma}

\begin{proof}
Recall that the fiber of $E^*$ above a point $L \subset Q$ is the space $\Hom(\bC^n/Q, L).$
The map $p_1$ is given by pre-composing each homomorphism $u: \bC^n / Q \to L$ with the quotient map $\bC^n/L\twoheadrightarrow\bC^n/Q.$
Thus its image consists exactly of maps
$\gamma:\bC^n/L\to L$ that vanish on $Q/L$.
Since
\[
T_L\Gr(k,Q)=\Hom(L,Q/L),
\]
this is equivalent to $\gamma$ annihilating $T_L\Gr(k,Q)$ under the trace
pairing.

Similarly, the map $p_2$ is given by post-composing each homomorphism with the inclusion
$L\hookrightarrow Q$. Its image consists of maps
$\eta:\bC^n/Q\to Q$
whose image is contained in $L$. Since
\[
T_Q\Gr(L,\ell)
=
\left\{
\beta\in\Hom(Q,\bC^n/Q)
\ :\
\beta|_L=0
\right\},
\]
this is equivalent to $\eta$ annihilating $T_Q\Gr(L,\ell)$ under the trace pairing.
\end{proof}

The following two lemmas tell us how the conormal varieties of the principal Chow--Lam locus and Chow--Lam discriminantal locus compare to the conormal variety of $\cV.$ They will be used in the proof of Theorem \ref{thm:rankk}.

\begin{lemma}\label{lem:pcl containment}
Let $\cV\subset\Gr(k,n)$ be an irreducible variety of dimension
$k(n-\ell)-1+i$, where $k< \ell < n$ and $0\leq i\leq k(\ell - k).$ Suppose that $\mathcal{PCL}^i_{\cV}$ is a hypersurface. Then
\begin{equation}\label{eq:pcl containment}
\Con(\mathcal{PCL}^i_{\cV})
\subset
p_2\left(
p_1^{-1}(\Con(\cV) )
\right).
\end{equation}
\end{lemma}
\begin{proof}
Let $\cH$ be an irreducible component of $\mathcal{PCL}^i_{\cV}$, and
let $\Phi_\cH \subset F$ be an irreducible component of the incidence variety
$\Phi^i_{\cV}$ which dominates $\cH$. Choose
a general point $(L,Q)\in\Phi_\cH$ such that
$L\in \cV_{\rm reg}, \,
Q\in \cH_{\rm reg}, \,
(L,Q)\in (\Phi_\cH)_{\rm reg},$
and the differential of the projection
$q:\Phi_\cH\longrightarrow \cH$ is surjective at $(L,Q)$. Indeed, each of these conditions is open, so we may make such a choice.

Since $\cV$ and $\Gr(k,Q)$ do not meet transversely at $L$, there is
a nonzero covector
\[
\phi\in
\Ann(T_L\cV)\cap\Ann(T_L\Gr(k,Q)).
\]
By Lemma~\ref{lem:eimage}, there is a point $u \in E^*_{(L,Q)}$ whose projection $p_1(u)$ equals $\phi.$ Set $\eta := p_2(u).$

We claim that $\eta$ is the unique conormal vector to $\cH$ at $Q$. Then $(Q, \eta)$ will lie in $\Con(\cH).$ Since every point of $\Con(\cH)$ arises in this manner, and lies in $p_2(p_1^{-1}(\Con(\cV)))$ by construction of $\eta$, the theorem statement follows.

To prove the claim, fix
$\beta\in T_Q\cH$. Then surjectivity of $dq$ gives a lift
\[
(\alpha,\beta)\in T_{(L,Q)}\Phi_\cH,
\qquad \alpha\in T_L\cV.
\]
Since $(\phi,\eta)$ is conormal to $F$ at $(L, Q)$, we have
\[
0  = \gen{(\phi, \eta), (\alpha, \beta)}  = \langle\phi,\alpha\rangle+\langle\eta,\beta\rangle.
\]
The term $\gen{\phi,\alpha}$ vanishes because $\phi$ is conormal to $\cV$. Thus
$\langle\eta,\beta\rangle=0$. Then $\eta$ annihilates $T_Q\cH$, so $(Q,\eta)$ lies in $\Con(\cH)$.
\end{proof}

The following example shows that the containment in Lemma \ref{lem:pcl containment} can be strict.

\begin{example}\label{eg:strict containment}
Let $\cV = V(q_{12}, q_{34}, q_{56})$ be the intersection of three Schubert divisors in $\Gr(2,6)$, as in Example \ref{eg:schubert reducible}. The conormal space at a general point $L$ is spanned by the covectors to the three divisors. Thus at a general point of $\cV$, the conormal to $\cV$ contains three covectors of rank one, but a general covector has rank two.

The principal Chow--Lam locus $\mathcal{PCL}^2_\cV$ is $V(p_{1234}p_{1256}p_{3456}) \subset \Gr(4,6).$ Its conormal variety is the union of three irreducible varieties corresponding to the three factors. Each Pl\"ucker monomial defines a hypersurface of conormal rank one. Thus $\Con(\mathcal{PCL}^2_\cV)$ is entirely contained inside the rank one locus $D_1 \subset T^*\Gr(4, 6).$ However, the right side of \eqref{eq:pcl containment} includes some rank two covectors, coming from the projection of $p_1^{-1}(N^*_\cV\Gr(k,n) \cap D_2^\circ),$ which is nonempty. Thus the containment in \eqref{eq:pcl containment} is strict.
\end{example}

It is instructive to understand the full rank stratification of the variety $\cV$ from the previous example, to see how the irreducible components of the principal Chow--Lam locus arise. 

 \begin{example}[Rank stratification of the previous example]\label{eg:conic rank stratification}
	Let $\cV$ be as in Example \ref{eg:schubert reducible}.
Recall the rank stratification $D_0 \subset D_1 \subset D_2.$ Set \[C_i := N^*_\cV \Gr(k,n) \cap D_i^\circ.\]
By \eqref{eq:cl fiber}, the fiber above a point in $C_0$ is isomorphic to $\Gr(2, 4).$ Similarly, the fiber above a point in $C_1$ is isomorphic to $\Gr(2,3),$ and the fiber above a point in $C_2$ is isomorphic to $\Gr(2,2).$ The fiber $p_1^{-1}(C_1)$ has at least three components, corresponding to the three covectors of rank one at a general point. Thus we obtain the dimension counts
\[\dim p_1^{-1}(C_0) = 5 + 0 + 4, \quad \dim p_1^{-1}(C_1) \leq 5 + 1 + 2, \quad  \dim p_1^{-1}(C_2) = 5 + 3 + 0.\]

However, there are three additional components of $p_1^{-1}(C_1)$ which come from lower-rank strata in $\cV.$ We examine one of these, namely:
\[\cW := \{L \in \cV \ : \ L \cap \overline{e_5e_6} \neq 0\}.\]
The variety $\cW$ has dimension four; each line is the span of a point in $\overline{e_1e_2e_3e_4}$ and a point in $\overline{e_5e_6},$ giving $3 + 1$ degrees of freedom.

Let $\phi_{12}, \phi_{34}, \phi_{56}$ be the covectors of the three Schubert varieties. Fix a general point $L \in \cW$ and let $p$ be the intersection point of $L$ with $\overline{e_5e_6}.$ The images of $\phi_{12}$ and $\phi_{34}$ are then precisely $\overline{e_3e_4e_5e_6} \cap L = \overline{e_1e_2e_5e_6} \cap L = p,$ so any linear combination of $\phi_{12}$ and $\phi_{34}$ has rank one. Therefore the conormal fiber $N^*_\cV\Gr(k,n)_L$ at a general point in $\cW$ has a two-dimensional subspace of rank one homomorphisms, contributing a component of $p_1^{-1}(C_1)$ of dimension $4 + 2 + 2 = 8.$
\end{example}

\begin{lemma}\label{lem:cl equality}
Let $\cV \subset \Gr(k,n)$ be an irreducible variety of dimension $k(n-\ell)-1 + i,$ where $k < \ell \leq n-k$ and $0 \leq i \leq k(\ell - k).$ Suppose that $\mathcal{CL}^i_\cV$ is a hypersurface. Then
\begin{equation}\label{eq:clequality}
\Con(\mathcal{CL}^i_{\cV})
=
\overline{
p_2\left(
p_1^{-1}(\Con(\cV)^\circ_k)
\right)
}.
\end{equation}
\end{lemma}
\begin{proof}
The same argument as in Lemma \ref{lem:pcl containment}, restricted to the rank $k$ locus, shows the left side of \eqref{eq:clequality} is contained in the right.

In the proof of Theorem \ref{thm:cldisc irreducible}, we showed that the incidence variety $\Psi_\cV^i$ is irreducible. Thus $p_2(\Psi_\cV^i)$ is irreducible of dimension at most $\dim\Gr(\ell,n)$. By Theorem \ref{thm:cldisc irreducible} $\mathcal{CL}^i_{\cV}$ is irreducible, so $\Con(\mathcal{CL}^i_{\cV})$ is irreducible, and has dimension $\dim \Gr(\ell, n)$. This forces equality in \eqref{eq:clequality}.
\end{proof}

We are now ready to prove the first main theorem from the introduction.

\begin{reptheorem}{thm:rankk}
Let $\cV\subset\Gr(k,n)$ be an irreducible variety of dimension
$k(n-\ell)-1+i$, where $k< \ell < n$ and $0\leq i\leq k(\ell - k).$ If $\mathcal{PCL}^i_{\cV}$ is a
hypersurface, then $\crk \mathcal{PCL}^i_{\cV} \leq \min(k, n-\ell).$
\end{reptheorem}

\begin{proof}
Let $(Q,\eta)$ be a general point of
$\Con(\mathcal{PCL}^i_{\cV})$. By
Lemma~\ref{lem:pcl containment}, there exist a flag $L\subset Q$ and a covector $u \in E^*_{(L,Q)}$
such that $\eta=p_2(u)$ in $\Hom(\bC^n/Q, Q).$ By Lemma~\ref{lem:eimage}, the image of the map $\eta$ lies in $L.$ But $L$ has dimension $k,$ so $\eta$ has rank at most $k$. Since $Q$ is in $\Gr(\ell, n),$ the rank is bounded above by $\ell, n - \ell$ as well, so the upper bound is $\min(k, n - \ell).$
\end{proof}


For the Chow--Lam discriminant, the upper bound in Theorem \ref{thm:rankk} is attained.

\begin{theorem}\label{thm:cl rank}
Let $\cV \subset \Gr(k,n)$ be an irreducible variety of dimension $k(n-\ell)-1 + i,$ where $k < \ell \leq n-k$ and $0 \leq i \leq k(\ell - k).$ If $\mathcal{CL}^i_{\cV}$ is a hypersurface, then $\crk \mathcal{CL}^i_{\cV} = k.$
\end{theorem}

\begin{proof}
Let $(Q, \eta)$ be a general point of $\Con(\mathcal{CL}^i_{\cV})$. By Lemma~\ref{lem:cl equality}, there exist a flag $L \subset Q$ and a covector $u \in  E^*_{(L,Q)}$ such that $\eta = p_2(u)$ and $p_1(u) \in \Con(\cV)^\circ_k.$ The covector $p_1(u)$ has rank exactly $k.$ Thus $\eta$ also has rank $k.$
\end{proof}

\section{Recovering a variety from a low-rank hypersurface}\label{sec:recovery}

The original motivation of Chow and van der Waerden for defining a Chow form of a projective variety was to capture this variety by a single polynomial. That is, given the Chow form ${\rm C}_\cV$, one can uniquely recover the underlying variety $\cV$ from it.

Gelfand, Kapranov, and Zelevinsky gave a strategy for recovering the variety $\cV$ from any of its associated hypersurfaces using symplectic geometry \cite{gkz}. In particular, their method shows that \emph{every} coisotropic hypersurface is an associated hypersurface of a projective variety. The goal of this section is to explain how to recover a subvariety of a Grassmannian from any of its Chow--Lam discriminants which are hypersurfaces.

The process is as follows. Given a covector $\eta:\bC^n/Q\to Q,$
let $L:=\im(\eta)$. Fix as before the notation $p: \bC^n / L \twoheadrightarrow \bC^n / Q.$ Then $L$ has dimension $k$ and $\eta$ induces a map
\[
\eta \circ p:
\bC^n/L\twoheadrightarrow\bC^n/Q
\to L.
\]
Then we obtain a map between cotangent bundles:
\begin{equation}\label{eq:conormal reduction}
\begin{split}
\rho: D^\circ_k & \to T^*\Gr(k,n) \\
(Q, \eta) & \mapsto (\im \eta,\eta \circ p).
\end{split}
\end{equation}
\begin{definition}
	Let $\cY\subset\Gr(\ell,n)$ be an irreducible variety of conormal rank $k$. We define the \emph{conormal reduction} of $\cY$ to be
\[
\Lambda_\cY
:=
\overline{\rho\bigl(\Con(\cY)^\circ_k\bigr)}
\subset T^*\Gr(k,n).
\]
\end{definition}

Let $\pi: T^*\Gr(k,n) \to \Gr(k,n)$ be the bundle map to the Grassmannian. Projecting the conormal reduction $\Lambda_\cY$ yields a subvariety $\pi(\Lambda_\cY)$ of $\Gr(k,n),$ which has rank $k$ by construction. We will prove later in Theorem \ref{thm:cl recovery converse} that we can recover a variety from its Chow--Lam discriminant via conormal reduction.

The key fact we will need is that the conormal reduction is in fact itself a conormal variety. To prove this, we first need a little bit of symplectic geometry. Let $X$ be a smooth complex manifold and let $\pi: T^*X \to X$ be the projection map from the cotangent bundle. Then $X$ possesses a tautological one-form on its cotangent bundle $T^*X,$ defined at every point $(x, \xi) \in T^*X$ by
\begin{equation*}
\begin{split}
\lambda_{(x, \xi)} : T_{(x, \xi)}(T^*X) & \to \bC \\
v & \mapsto \xi(d\pi_{(x, \xi)}(v)).
\end{split}
\end{equation*}
In local coordinates $(x_1, \ldots, x_n, \xi_1, \ldots, \xi_n)$ on $T^*X,$ the form $\lambda_X$ may be expressed as
\[\lambda_X := \sum_i \xi_i d x_i.\]
Let $\omega_X$ denote the symplectic form on $T^*X,$ which is a two-form. Then we have the relation
\[\omega_X = \sum_{i = 1}^n   d x_i \wedge d \xi_i = -d \lambda_X.\]

The following result may be known in the literature, but we could not find it in the form we wanted, so we provide a self-contained proof for the convenience of the reader.

\begin{lemma}\label{lem:tautological}
Let $\lambda_{\ell}, \lambda_k$ be the tautological one-forms on $T^*\Gr(\ell, n)$ and $T^*\Gr(k,n),$ respectively. Let $\rho$ be defined as in \eqref{eq:conormal reduction}. Then $\rho^* \lambda_k = \lambda_{\ell}|_{D^\circ_k}.$
\end{lemma}
\begin{proof}
Let $\lambda_k$ be the tautological one-form on the Grassmannian $\Gr(k, n)$ and let \[\pi_k : T^*\Gr(k, n) \to \Gr(k,n)\] be the bundle map. Since the duality between tangent and cotangent spaces is given by the trace map, we have
\[(\lambda_k)_{(L, \phi)}(v) = {\rm tr}(\phi \circ (d\pi_k)_{(L, \phi)}(v)),\]
where $v$ is in $T_{(L, \phi)}(T^*\Gr(k,n))$ and $(d\pi_k)_{(L, \phi)}(v)$ is in $T_L\Gr(k,n).$ Suppose $(Q, \eta)$ is in $D^\circ_k$ and define $(L, \eta \circ p)$ to be the image of the map $\rho$ as in \eqref{eq:conormal reduction}.
Then we have
\begin{equation}\label{eq:rho lambdak} 
	\begin{split}
	(\rho^*\lambda_k)_{(Q, \eta)}(v) & = (\lambda_k)_{(L, \eta \circ p)}(d \rho_{(Q, \eta)}(v))\\
& = {\rm tr}(\eta \circ p \circ d(\pi_k \circ \rho)_{(Q, \eta)}(v))
	\end{split}
\end{equation}
where the first equality is by definition of pullback of one-forms and the second is by functoriality of the differential. Fix the notation
\[
\dot{L} := d(\pi_k \circ \rho)_{(Q,\eta)} \colon T_{(Q,\eta)}D^\circ_k \to T_L\Gr(k,n),
\qquad
\dot{Q} := (d\pi_\ell)_{(Q,\eta)} \colon T_{(Q,\eta)}D^\circ_k \to T_Q\Gr(\ell,n).
\]
Define the auxiliary map
\[f  = (\pi_k \circ \rho, \pi_\ell): D^\circ_k \to \Fl(k, \ell, n),\]
and observe its differential is
\[df = (d(\pi_k \circ \rho), d \pi_\ell): T_{(Q, \eta)}D^\circ_k \to T_{(L, Q)}\Fl(k, \ell, n).\]
Then $df_{(Q, \eta)}(v)$ lies in $T_{(L, Q)} F$ and its projections to $T_L\Gr(k,n)$ and $T_Q\Gr(\ell, n)$ are $\dot{L}$ and $\dot{Q},$ respectively. But by the tangent-space description of $F$ we have
\[p \circ \dot{L}(v) = \dot{Q}(v)|_L \ \in \Hom(L, \bC^n/Q).\]
The image of $\eta \circ \dot{Q}(v)$ lies in $L$. Thus the traces satisfy
\begin{equation}\label{eq:trace equality}
{\rm tr}((\eta \circ p) \circ \dot{L}(v)) = {\rm tr} (\eta \circ (p \circ \dot{L}(v))) = {\rm tr} (\eta \circ \dot{Q}(v)).
\end{equation}
The left side of \eqref{eq:trace equality} is \eqref{eq:rho lambdak} and the right side is precisely $(\lambda_\ell)_{(Q, \eta)}(v).$
\end{proof}

We are now able to prove that the conormal reduction is itself a conormal variety.
\begin{proposition}\label{prop:reduction conormal}
Let $\cY\subset\Gr(\ell,n)$ be an irreducible rank $k$ variety. Its conormal reduction $\Lambda_\cY$ is an irreducible conic Lagrangian
subvariety of $T^*\Gr(k,n).$ In particular, $\Lambda_\cY=\Con(\pi(\Lambda_\cY)),$ where $\pi$ is the bundle projection to the Grassmannian $\Gr(k,n)$.
\end{proposition}
\begin{proof}
Since $\crk \cY = k,$ we have that $\Con(\cY)^\circ_k$ is a nonempty open subset of $\Con(\cY),$ hence irreducible. Thus the closure $\Lambda_\cY$ of its image under $\rho$ is also irreducible.

The variety $\Lambda_\cY$ is conic with respect to the $\bC^*$-action on the cotangent bundle. Indeed, $(\im \eta, t \cdot (\eta \circ p))$ is the image of $(Q, t \cdot \eta),$ so it lies in $\Lambda_\cY.$

Now, let us compute the dimension of $\Lambda_\cY.$ The fiber of $\rho$ in $D_k^\circ$ consists of
\[\rho^{-1}(L, \phi) = \{(Q, \eta): Q / L \subset \ker \phi\},\]
and $\eta$ is determined by $\phi.$ This has dimension $ (\ell - k)(n - \ell - k).$ Now, fix the quotient $\pi_Q: \bC^n \twoheadrightarrow \bC^n / Q.$ At a point $\rho(Q, \eta)$, the fiber consists of
\[\rho^{-1}(\rho(Q, \eta)) \cap N^*_\cY \Gr(\ell,n) \subset \{(Q', \eta'): \im \eta \subset Q' \subset \pi_Q^{-1}(\ker \eta)\},\]
where $\eta'$ is determined by $Q'.$ Thus the fiber in $N^*_\cY \Gr(\ell,n) $ has dimension $ \leq (\ell - k)(n - \ell - k),$ so the image $\Lambda_\cY$ has dimension $\geq k(n-k).$

To show that $\Lambda_\cY$ is Lagrangian, it remains to show that the symplectic form $\omega_{\Gr(k,n)}$ vanishes on it. This will also yield the upper bound $\dim \Lambda_\cY \leq k(n-k).$ But by Lemma \ref{lem:tautological}, we know that $\rho^*\lambda_k = \lambda_\ell.$ Since $\Con(\cY)$ is a conormal variety, $\lambda_\ell$ vanishes on it. Thus $\omega_{\Gr(k,n)} = - d\lambda_k$ is zero on the image $\rho(\Con(\cY)^\circ_k).$  So $\Lambda_\cY$ is Lagrangian, as desired.
\end{proof}

We now prove the second main theorem from the introduction, which states that every sufficiently low-rank hypersurface is a Chow-Lam discriminant. Note that the index $i$ is determined by the conormal reduction.

\begin{reptheorem}{thm:cl recovery}
Let $\cH \subset \Gr(\ell,n)$ be a hypersurface of conormal rank $k$, where $k < \ell.$ Let $\cV := \pi(\Lambda_\cH)$ be the projection of its conormal reduction to $\Gr(k,n)$. If $\dim \cV = k(n - \ell) - 1+i$ for some $0 \leq i \leq k(\ell - k),$ then $\cH$ equals the Chow--Lam discriminant $\mathcal{CL}^i_\cV$ of $\cV.$
\end{reptheorem}
\begin{proof}
Since every variety is the projection of its conormal variety, it suffices to show that \begin{equation}\label{eq:conormals equal}
\Con(\cH) = \overline{p_2(p_1^{-1}(\Con(\cV)^\circ_k))}.
\end{equation} Suppose that $(Q, \eta)$ is a general point of $\Con(\cH),$ so that $\eta$ has rank $k.$ Then $\rho(Q, \eta) = (L, \phi)$ is in $\Con(\cV)$ by Proposition \ref{prop:reduction conormal}. Furthermore, $\eta$ viewed as a map to its image gives $u: \bC^n / Q \to L$, with $p_1(u) = \phi$ and $p_2(u) = \eta.$ Thus the left side is contained in the right, with equality since they are irreducible of the same dimension.
\end{proof}

\begin{theorem}\label{thm:cl recovery converse}
Let $\cV \subset \Gr(k,n)$ be a variety of conormal rank $k$ and dimension $k(n - \ell) - 1+i$ for some $0 \leq i \leq k(\ell - k)$ and $k < \ell.$ Suppose that $\mathcal{CL}^i_\cV$ is a hypersurface. Then $\pi(\Lambda_{\mathcal{CL}^i_\cV})$ equals $\cV.$
\end{theorem}

\begin{proof}
We use Lemma \ref{lem:cl equality} again. Set $\cH=\mathcal{CL}^i_\cV$.
For
\[
(L,\phi)\times(Q,\eta)
\in
p_1^{-1}(\Con(\cV)^\circ_k),
\]
write $u:\bC^n/Q\to L$ for the corresponding element of $N_F^*X$.
By \eqref{eq:conormal projections}, the projections are
\[
\phi:
\bC^n/L\twoheadrightarrow\bC^n/Q\xrightarrow{u}L,
\qquad
\eta:
\bC^n/Q\xrightarrow{u}L\hookrightarrow Q.
\]
Since $\rk\phi=k=\dim L$, we have $\rk\eta=k$ and
$\im \eta=L$. Hence, by the definition of $\rho$,
\[
\rho(Q,\eta)=(L,\phi).
\]
Thus
\[
\rho\left(
p_2\left(
p_1^{-1}(\Con(\cV)^\circ_k)
\right)\right)
=
\Con(\cV)^\circ_k.
\]
The left-hand side is dense in $\Lambda_\cH$, while the right-hand
side is dense in $\Con(\cV)$ because $\cV$ has conormal rank $k$.
Taking closures gives $\Lambda_\cH=\Con(\cV).$
Finally, the projection of $\Con(\cV)$ to $\Gr(k,n)$ is $\cV$.
\end{proof}

\begin{remark}
The recovery process given by Theorem \ref{thm:cl recovery converse} is different from the \emph{recovered variety} described in \cite{PrattRanestad}. That will, given the ordinary Chow--Lam form of a variety $\cV \subset \Gr(k,n)$, produce a variety containing $\cV$ but which may have extra components. For instance, when the variety $\cV$ is the intersection of three Schubert hyperplanes in $\Gr(2,5)$, the recovered variety in $\Gr(2,5)$ will contain $\cV$ but also six other irreducible components, three of dimension $2$ and three of dimension $3$ \cite[Example 6.1]{PrattRanestad}. The conormal reduction, on the other hand, uniquely recovers the original variety $\cV.$
\end{remark}

From this section and the previous section, one may reach the conclusion that the Chow-Lam discriminant is the ``correct'' object to study; it is irreducible and completely encodes the variety (Theorem \ref{thm:cl recovery converse}). However, there are advantages to both constructions. For example, the principal Chow-Lam locus is easier to calculate, because the incidence variety used is a subvariety of $\Fl(k, \ell, n)$ rather than of the conormal bundle $N^*_F\Gr(k,n)$. Thus encoding it in a computer algebra system such as \texttt{Macaulay2} requires fewer variables and runs in less time. Even if one is interested in the Chow-Lam discriminant, it is often easier to compute the Chow-Lam locus and then extract the discriminant as an irreducible factor. 

Furthermore, the bounds on $k, \ell,$ and $n$ are more generous in the case of the principal Chow-Lam form, allowing more varieties as possible input. Indeed, in the next section we see examples of varieties that have (ordinary, irreducible) Chow-Lam forms but no Chow-Lam discriminants (Example \ref{eg:positroid38 ranks}). This example also shows that Chow-Lam forms of positroid varieties may not be Chow-Lam discriminants. Thus to describe polynomials appearing in nature, it is convenient to have both constructions. 

\section{Positroid Varieties}\label{sec:positroids}
In this section we study positroid varieties and their ranks. The \emph{matroid} $M_L$ of a point $L \in \Gr(k,n)$ is the set of indices $I \in \binom{[n]}{k}$ such that the Pl\"ucker coordinate $p_I(L)$ is nonzero. The set of indices $I \in \binom{[n]}{k}$ not in the matroid are called the \emph{nonbases} of the matroid. A matroid $M$ is called a \emph{positroid} if $M = M_L$ for a point $L$ whose Pl\"ucker coordinates are nonnegative real numbers.

The \emph{realization space} of a matroid $M$ is the set of points $\{L \in \Gr(k,n) \ : \ M = M_L\}.$ The Zariski closure of the realization space is called the \emph{matroid variety} of $M,$ or \emph{positroid variety} if $M$ is a positroid. We denote it by $\Pi_M.$ The positroid variety $\Pi_M$ is cut out, even scheme-theoretically, by the Pl\"ucker coordinates indexed by nonbases \cite[Theorem 5.15]{KnutsonLamSpeyer}; this fails for general matroid varieties. 

Chow--Lam forms were originally studied for positroid varieties under the name \emph{universal amplituhedron variety} \cite[Section 18.1]{Lam}. They arose in the context of scattering amplitudes for certain quantum field theories. In many cases, they exhibit interesting positivity properties. They have been conjectured to be elements of Lusztig's dual canonical basis \cite[Conjecture 19.8]{Lam}, and have also been shown to be cluster variables in some cases; see \cite[Theorem 8.17]{EvenZohar}. Theorem \ref{thm:rankk} can be applied to decide when a polynomial in Pl\"ucker coordinates of $\Gr(\ell, n)$, such as a particular cluster variable, is a principal Chow--Lam form of some subvariety of $\Gr(k,n)$ for $k<\ell$.

\begin{example}
Fix $k = 2, \, n = 9$ and let $M$ be the matroid whose nonbases are $\{12,13,23,45,67,89\}.$ Let $\Pi_M$ denote the corresponding positroid variety. The Chow--Lam form was computed in \cite[Section 19.4]{Lam}:
\[CL_{\Pi_M} = p_{1235}p_{1237}p_{4689}
- p_{1234}p_{1237}p_{5689}- p_{1235}p_{1236}p_{4789}
+ p_{1234}p_{1236}p_{5789}.
\]
Theorem \ref{thm:rankk} predicts that the conormal rank is at most two.
Indeed, one can verify by explicit computation that it is exactly two. This can be done symbolically but in this case it is easier to sample a point on the vanishing locus of $CL_{\Pi_M}.$ Indeed, it suffices to sample a two-dimensional subspace of $\bC^9$ from the positroid variety, and take the direct sum with a general two-dimensional subspace.

Contrast this with the following degree three cluster variable for $k=3$ and $n=8$:
\[A := p_{134}\left(p_{258}p_{167}
- p_{678}p_{125}\right)
- p_{158}p_{234}p_{167}.
\]
This is the cluster variable $A$ from \cite[Theorem 8]{scott}. Again, one can compute explicitly that the conormal rank of the hypersurface defined by this polynomial is three, by taking the conormal vector at a general point. Thus it is not a Chow--Lam form or principal Chow--Lam form of any variety in $\Gr(2,8)$ or $\Gr(1,8) = \bP^7.$
\end{example}

We describe the conormal space to a positroid variety at a general point, and use this to calculate the conormal rank. In Section \ref{subsec:conormal rank} we carried out the conormal rank calculation for Schubert varieties and torus orbit closures. Our strategy for each of these was to parameterize a dense open set. While there are such parameterizations available for open subsets of positroid varieties by work of Postnikov \cite{postnikov}, we find it more convenient to work with the defining ideal directly.

Given a linear space $L,$ a \emph{dependent set} of $L$ is a nonempty set $C = \{i_1, \ldots, i_r\}\subset [n]$ such that the projection of $L$ to $\text{span}(e_{i_1}, \ldots, e_{i_r})$ has dimension less than $r.$ A \emph{circuit} is a minimal dependent set. If the linear space $L$ is given as the rowspan of a matrix $X$, then the \emph{circuit vector} corresponding to $C$ is the unique (up to scalar) vector in the kernel of $X$ with support $C.$ In other words, it yields the unique linear dependency between the columns indexed by $C.$

We define a new matroid, whose ground set is a subset of the nonbases of $M$, namely
\begin{equation*}
E_M := \{I \in \binom{[n]}{k} : \rk  I = k - 1\}.
\end{equation*}
Each subset $I$ of rank $k-1$ contains a unique circuit $C_I.$ We define the matroid $M_{\bf circ}$ to be the linear matroid of the corresponding circuit vector $v_I$ in $\bC^n.$ Note that $v_I$ naturally lives in the perpendicular space $L^\perp \subset \bC^n.$

Now, since $I$ has rank $k-1$, the span of the columns $\{X_i : i \in I\}$ is a hyperplane in $\bC^k.$ It is cut out by a unique (up to scalar) linear functional $\lambda_I,$ which we may view as a vector of length $k.$ The \emph{cocircuit vector} $w_I$ corresponding to $I$ is the vector $\lambda_I X$ in $\bC^n$ resulting from matrix multiplication. Note that $w_I$ naturally lives in the vector space $L.$ We define the matroid $M_{\bf cocirc}$ to be the linear matroid on the cocircuit vectors. We then have the following result.

\begin{proposition}\label{prop:positroid rank}
Let $M$ be a positroid and let $\Pi_M$ be the corresponding positroid variety. The conormal space to $\Pi_M$ at a point $L$ with $M = M_L$ has rank equal to
\[\max \{|S| : S \subseteq E_M \text{ is independent in both } M_{\bf circ} \text{ and } M_{\bf cocirc}\}.\]
\end{proposition}
\begin{proof}
Recall \(\Pi_M\) is cut out in the Grassmannian scheme-theoretically  by the Pl\"ucker coordinates corresponding to nonbases of \(M\) \cite[Theorem 5.15]{KnutsonLamSpeyer}. Thus the differentials of the coordinates $(dp_I)_L$, as $I$ ranges over nonbases of $M$, span the cotangent space at $L.$ For each nonbasis $I $, let $v_I \in L^\perp$ be the circuit vector and $w_I \in L$ be the cocircuit vector. We claim that under the identification $\Hom(\bC^n / L, L) = L^\perp \otimes L,$ the conormal space to $\Pi_M$ at $L$ is
\begin{equation}\label{eq:cotangent positroid}
N^*_{\Pi_M}\Gr(k,n)_L = \mathrm{span} \{v_I \otimes w_I : I \in E_{M}\}.
\end{equation}

We may prove \eqref{eq:cotangent positroid} by working in an affine chart for each nonbasis $I.$ Without loss of generality, suppose $I = \{1, \ldots, k\}.$ Let $r$ be the rank of $I,$ and choose a maximal independent set $J \subset I$ of the vectors. Let us parameterize the linear space $L$ by a $k \times n$ matrix $X.$ We work in an affine chart in which the columns of $X$ indexed by $J$ are the basis vectors $e_1, \ldots, e_r \in \bC^k$. Then we have $n-r$ remaining columns filled with variables $a_{11}, \ldots, a_{k,n-r}$. In particular, \[dp_I / da_{ij} = \det A\]
where $A$ is the square minor of $X$ with rows $[r, k] \setminus i$ and columns $[r, k] \setminus j.$ At $L,$ every entry of $A_{[r, k]}$ vanishes. In particular, if $I$ has rank less than $k-1$, then the vector $(dp_I)_L  $ is zero.

Suppose that $I$ has rank $k-1.$ Then $(dp_I)_L  = da_{k,k}.$ Inside $\bC^k,$ the hyperplane cutting out the span of the $I$ vectors is precisely $e_k^*.$ 
Thus under the identification $L^\perp \otimes L,$ we get
 \[v_I \otimes w_I = da_{k,k}.\]
In particular, the cotangent space is spanned by a collection of rank one matrices. The maximum rank of a sum of these matrices is the maximum size of a nonzero minor of that sum. Write $\sum_{I} c_I v_I \otimes w_I = V \cdot {\rm diag}(c) \cdot W^T.$ By the Cauchy--Binet theorem, an $s \times s$ minor of the sum indexed by rows $R,$ columns $T$ is given by
\[\det (V \cdot {\rm diag}(c) \cdot W^T)_{R, T} = \sum_{|S| = s} \det V_{R, S} \det W_{T,S} \prod_{I \in S}c_I.\]
The products in the $c_I$ are square-free and distinct. Thus there is no cancellation and the sum is nonzero whenever at least one term $\det V_{R, S} \det W_{T,S}$ is nonzero, or equivalently when the vectors $(v_{I})_{I \in S}$ and $(w_{I})_{I \in S}$ are independent.
\end{proof}

\begin{remark}\label{rmk:cocircuit}
	The matroids $M_{\bf circ}$ and $M_{\bf cocirc}$ depend not only on the matroid $M,$ but on the underlying linear space. For example, consider the matrices
	\[B=
\begin{pmatrix}
0&0&0&1&1&1&1&1&1\\
1&1&1&0&0&0&5&9&3\\
9&5&8&8&7&5&0&0&0
\end{pmatrix},
\qquad
B'=
\begin{pmatrix}
0&0&0&1&1&1&1&1&1\\
1&1&1&0&0&0&-1&-1&-1\\
7&5&9&1&3&5&2&8&5
\end{pmatrix}.
\]
Then the matrices have the same set $\{123,456,789\}$ of nonbases. However, the cocircuit vectors are
\[e_1B, \, e_2B, \, e_3B \in L, \quad e_1B', \, e_2B', \, (e_1 + e_2)B' \in L'.\]
Thus the cocircuit matroid of $L$ is $U_{3,3}$ whereas the cocircuit matroid of $L'$ is $U_{2,3}.$

However, the matroids $M_{\bf circ}$ and $M_{\bf cocirc}$ will be the same for any \emph{generic} member $L$ of the realization space $\Pi_M$, where generic here is a stronger condition than simply $M_L = M.$ Indeed, the realization space of a positroid is irreducible, and the linear dependency of the cocircuit vectors is a closed condition on elements of $\Pi_M$.
\end{remark}

The situation for positroids for $k=2$ is particularly simple to describe. The following can in principle be proven from Proposition \ref{prop:positroid rank} but it is easier to prove directly.
\begin{lemma}
Let $M$ be a loopless positroid of rank two. Then $\Pi_M$ has conormal rank two if and only if there exist two nonbases not contained in the same parallel class.
\end{lemma}
\begin{proof}
Suppose there are two such nonbases. Then the ground set of $M$ has size at least four. By relabelling, we may suppose that a point $L \in \Pi_M$ has a matrix representative of the form
\[\begin{bmatrix}
1 & a & 0 & 0 & \ldots \\
0 & 0 & 1 & b & \ldots
\end{bmatrix}\]
for some $a, b.$ Then a choice of a conormal vector in the affine chart $p_{13} \neq 0$ is $dp_{12} + dp_{34}.$
For the other direction, if all nonbases are contained in the same parallel class, then points are locally parameterized by
\[\begin{bmatrix}
1 & 0 & 0 & 0 & \ldots \\
0 & 1 & a & b & \ldots
\end{bmatrix}\]
and every conormal vector has rank at most one.
\end{proof}

As $k$ gets higher, there are more possibilities for the conormal rank.
\begin{example}\label{eg:positroid38 ranks}
We here record the ranks of codimension $4$ positroids for $k = 3, n=6,$ up to permutations of the ground set $[6].$ The entries $\star$ in the columns labeled nonbases may be filled with any numbers. The column labelled $f$ records the bounded affine permutation of the positroid variety; see \cite[Section 6]{Lam}. One can use this to compute the cohomology class of the positroid variety via the algorithm in \cite[Section 10]{Lam}. Here $\sigma_\lambda$ denotes the cohomology class of the Schubert variety $\Omega_\lambda$.

\begin{table}[ht]
\centering
\small
\renewcommand{\arraystretch}{1.3}
\setlength{\tabcolsep}{3pt}

\begin{tabularx}{\textwidth}{@{}
    >{\raggedright\arraybackslash}X
    >{\centering\arraybackslash}p{0.13\textwidth}
    >{\centering\arraybackslash}p{0.19\textwidth}
    >{\centering\arraybackslash}p{0.10\textwidth}
@{}}
\hline
Nonbases of $M$
& $f$
& $[\Pi_M]$
& $\operatorname{crk}(\Pi_M)$\\
\hline

$156,234,235,236,245,246,256,345,346,356,456$
& $[745869]$
& $\sigma_{31}$
& $2$
\\[2mm]

$\star 45, \star 46, \star 56$
& $[478569]$
& $\sigma_{22}$
& $2$
\\[2mm]

$\star \star 6, 345$
& $[475896]$
& $\sigma_{211}$
& $2$
\\[2mm]

$125,126,156,256,345,346,356,456$
& $[475968]$
& $\sigma_{31}+\sigma_{22}$
& $2$
\\[2mm]

$\star 43, \star 56$
& $[574869]$
& $\sigma_{22}+\sigma_{211}$
& $2$
\\[2mm]

$126,136,146,156,234,235,245,345$
& $[645897]$
& $\sigma_{31}+\sigma_{211}$
& $2$
\\[2mm]

$126,145,245,345,346,356,456$
& $[476598]$
& $\sigma_{31}+\sigma_{22}+\sigma_{211}$
& $3$
\\

\hline
\end{tabularx}

\caption{The codimension-$4$ positroid varieties in $\Gr(3,6)$,
up to symmetry.}
\label{tab:gr36-codim4}
\end{table}

Each of the positroid varieties listed in Table \ref{tab:gr36-codim4} has a  Chow--Lam form, which has degree equal to the coefficient of the class $\sigma_{31}$ \cite[Theorem 3.5]{ChowLam}. However, in this case there are no Chow--Lam discriminants, since $\ell = 4 > n - k = 3.$ Indeed, the Chow--Lam discriminant would have to have rank $3$ by Theorem \ref{thm:cl rank} and live in $\Gr(4,6),$ which is impossible.

The Python code for this computation was produced by ChatGPT 5.6. It enumerates the positroids up to symmetry, converts between Grassmann necklace and affine permutation indexing, and uses \texttt{stanley\_symmetric\_function()} from \texttt{Sage} to obtain the cohomology class. To compute the conormal rank, it produces a matrix representative from the Le-diagram parameterization and uses Proposition \ref{prop:positroid rank} to get a lower bound. The upper bound is achieved as follows. For every $S \subset E_M,$ set $U_S = \cup_{I \in S} C_I.$ Then
\begin{equation}\label{eq:cocircuit bound}
\operatorname{rk}\{v_I : I \in S\}
\le \dim L^\perp \cap \bC^{U_S} = 
|U_S|-\operatorname{rk}_M(U_S).
\end{equation}
Then \eqref{eq:cocircuit bound} gives an upper bound on the rank of the circuit matroid, which bounds the conormal rank by Proposition \ref{prop:positroid rank}. The author cross-checked by hand that the matrix representatives (listed in Appendix \ref{app:example75-computation}) give the correct nonbases and bounded affine permutations. 
\end{example}

In this article our main goal was to develop the theory of conormal rank, generalizing the study of coisotropic hypersurfaces of Gelfand--Kapranov--Zelevinsky. Thus there are many natural questions which we have not addressed. In particular, it would be interesting to see whether the statistic of conormal rank has some combinatorial meaning for other classes of polynomials in Pl\"ucker coordinates, such as web invariants or cluster variables.

\subsection*{LLM Use}
ChatGPT 5.6 Sol was used to generate mathematical content in Propositions \ref{prop:schubert rank}, \ref{prop:torus orbit rank}, \ref{prop:positroid rank}, Lemmas \ref{lem:con vs bundle}, \ref{lem:tautological}, Examples \ref{eg:conics}, \ref{eg:conic rank stratification}, \ref{eg:positroid38 ranks} and the example of Remark \ref{rmk:cocircuit}. It was also used to shorten and edit various proofs in Sections 4, 5, and 6. ChatGPT suggested the name ``conormal reduction''.
Claude Fable generated a referee report which pointed out many minor errors. The text itself was written solely by the author, except for typesetting formulas, figures, and tables, and formatting citations. The author takes full responsibility for the contents of the paper.

\subsection*{Acknowledgements}
We thank Kathl\'{e}n Kohn for suggesting this problem and for helpful discussions. The author received support from NSF GRFP no. 2023358166, a UC Davis Chancellor's Fellowship, and a Perimeter Institute Postdoctoral Fellowship while this work was being conducted.

\appendix
\section{Matrix representatives for Table \ref{tab:gr36-codim4}}
\label{app:example75-computation}

For the seven classes appearing in Example~7.5, in the order in which they
occur in Table~\ref{tab:gr36-codim4}, we use the following representatives in
$\Gr(3,6)$:
\[
A_1=
\begin{pmatrix}
1&0&0&0&0&0\\
0&1&0&-1&-2&-2\\
0&0&1&1&1&1
\end{pmatrix},
\qquad
A_2=
\begin{pmatrix}
1&0&0&1&1&1\\
0&1&0&-1&-1&-1\\
0&0&1&1&1&1
\end{pmatrix},
\]
\[
A_3=
\begin{pmatrix}
1&0&0&1&2&0\\
0&1&0&-1&-2&0\\
0&0&1&1&1&0
\end{pmatrix},
\qquad
A_4=
\begin{pmatrix}
1&0&0&1&1&1\\
0&1&0&-1&-1&-1\\
0&0&1&1&0&0
\end{pmatrix},
\]
\[
A_5=
\begin{pmatrix}
1&0&0&0&1&1\\
0&1&0&0&-1&-1\\
0&0&1&1&1&1
\end{pmatrix},
\qquad
A_6=
\begin{pmatrix}
1&0&0&0&0&1\\
0&1&0&-1&-2&0\\
0&0&1&1&1&0
\end{pmatrix},
\]
\[
A_7=
\begin{pmatrix}
1&0&0&1&1&1\\
0&1&0&-1&-1&-1\\
0&0&1&1&1&0
\end{pmatrix}.
\]

	\bigskip
	\bigskip

	\footnotesize
	\noindent {\bf Author's address:}

	\smallskip

	\noindent Elizabeth Pratt, Perimeter Institute for Theoretical Physics
	\hfill \url{lpratt@perimeterinstitute.ca}

\end{document}